\documentclass[11pt]{article}
\newcommand{\authorfootnotes}{\renewcommand\thefootnote{\@fnsymbol\c@footnote}}%
\usepackage{amssymb,amsmath, amsfonts, amsthm}

\usepackage{graphicx}
\usepackage{tikz}
\usepackage{enumerate}

\newcommand{\cp}{\,\square\,}

\newcommand{\vertex}{\node[vertex]}
\tikzstyle{vertex}=[circle, draw, inner sep=0pt, minimum size=6pt]

\usetikzlibrary{arrows}

\newtheorem{thm}{Theorem}[section]
\newtheorem{lem}[thm]{Lemma}
\newtheorem{cor}[thm]{Corollary}
\newtheorem{prop}[thm]{Proposition}
\newtheorem{prob}{Problem}

\newcommand{\D}{{\rm DOM}}

\begin{document}

\title{Orientable domination in product-like graphs}

\author{
Sarah Anderson$^{a}$
\and
Bo\v{s}tjan Bre\v{s}ar$^{b,c}$
\and
Sandi Klav\v{z}ar$^{d,b,c}$
\and
Kirsti Kuenzel$^{e}$
\and
Douglas F. Rall$^{f}$\\
}

\date{\today}

\maketitle

\begin{center}
$^a$ Department of Mathematics, University of St. Thomas, St. Paul, MN, USA\\
$^b$ Faculty of Natural Sciences and Mathematics, University of Maribor, Slovenia\\
$^c$ Institute of Mathematics, Physics and Mechanics, Ljubljana, Slovenia\\
$^d$ Faculty of Mathematics and Physics, University of Ljubljana, Slovenia\\
$^e$ Department of Mathematics, Trinity College, Hartford, CT, USA\\
$^f$ Department of Mathematics, Furman University, Greenville, SC, USA\\
\end{center}

\begin{abstract}
The orientable domination number, ${\rm DOM}(G)$, of a graph $G$ is the largest domination number over all orientations of $G$. In this paper, ${\rm DOM}$ is studied on different product graphs and related graph operations. The orientable domination number of arbitrary corona products is determined, while sharp lower and upper bounds are proved for Cartesian and lexicographic products.  A result of Chartrand et al.\ from 1996 is extended by
establishing the values of ${\rm DOM}(K_{n_1,n_2,n_3})$ for arbitrary positive integers $n_1,n_2$ and $n_3$. While considering the orientable domination number 
of lexicographic product graphs, we answer in the negative a question concerning domination and packing numbers in acyclic digraphs posed in [Domination in digraphs and their direct and Cartesian products, J.\ Graph Theory 99 (2022) 359--377].
\end{abstract}


\noindent
{\bf AMS Subj.\ Class.\ (2020)}: 05C20, 05C69, 05C76.

\maketitle

\section{Introduction}

Domination is one of the most explored topics in graph theory. On the other hand, this concept has not received as much attention in directed graphs. There are several ways in which domination in graphs transfers to directed graphs, notably in-domination, out-domination, twin domination, and reverse domination. The most standard one, however, is out-domination, which is then referred to simply as domination. For some recent papers on domination in digraphs see papers~\cite{dong-2015,hao-2018,hlnt-2018} and a recent survey~\cite{bookHHH}.

Chartrand, VanderJagt and Yue defined and studied two invariants on undirected graphs that are based on the domination number of the orientations of the graph~\cite{cvy-1996}.
The concept presented in~\cite{cvy-1996}, which is the main topic of this paper, is defined as follows. Given an undirected graph $G$, the {\em orientable domination number} of $G$ is $$\D(G)=\max\{\gamma(D):\, D\textrm{ is an orientation of }G\}.$$
Replacing $\max$ with $\min$ in the above definition gives the (ordinary) domination number $\gamma(G)$ of $G$. The first result on orientable domination number, although stated in a different language, was proved by Erd\H{o}s~\cite{erdos} who found the following bounds for $\D(K_n)$, where $n\ge 2$: $\log_2 n-2\log_2(\log_2 n)\le \D(K_n)\le \log_2(n+1).$
Interestingly, the exact values of the orientable domination number of complete graphs are still not known. Szekeres and Szekeres~\cite{szekeres-1965}
improved the upper bound of Erd\H{o}s to arrive at the following result:
\begin{equation}
\label{eq:log}
\log_2 n-2\log_2(\log_2 n)\le \D(K_n)\le \log_2 n-\log_2(\log_2 n)+2.
\end{equation}
Lu, Wang and Wong~\cite{lu-2000} gave a short proof of the above upper bound.  Dominating sets
in tournaments were studied in~\cite{reid-2004} and used in~\cite{alon-2006}.

One of the main open problems in domination in graphs is Vizing's conjecture from~\cite{v1968}, which considers domination in the Cartesian product of graphs (see the survey paper~\cite{BDGHHKR2012} and the references therein).  A number of papers are devoted not only to the conjecture itself, but also to several other domination invariants in various graph products. With this paper we initiate the study of orientable domination in graph products and with respect to related graph operations.

In the next section, we establish notation and mention some preliminary results needed throughout the paper. In Section~\ref{sec:cor}, we determine the orientable domination number of the corona $G\odot H$ of arbitrary graphs $G$ and $H$, which is expressed as a function of the orientable domination numbers of $G$ and $H$. In Section~\ref{sec:cart}, we prove a sharp lower and upper bound on the orientable domination number of the Cartesian product of two graphs $G$ and $H$. We also pose the Vizing-like problem, whether $\D(G\cp H)\ge \D(G)\D(H)$ holds for all graphs $G$ and $H$, and observe that it holds if at least one of $G$ or $H$ is bipartite.  In Section~\ref{sec:lex}, we prove a sharp lower and  upper bound on $\D(G\circ H)$, where $G\circ H$ is the lexicographic product of graphs $G$ and $H$.
We continue with generalized lexicographic products, where the main focus is given to complete multipartite graphs. We establish the values of $\D(K_{n_1,n_2,n_3})$ for arbitrary positive integers $n_1,n_2$ and $n_3$, by which we extend a result of Chartrand et al.~\cite{cvy-1996}.
We also study the orientable domination number in specific classes of lexicographic product graphs.  In particular, a specific orientation of the graph $C_{2k+1}\circ \overline{K_s}$ allows us to provide a negative answer to the problem from~\cite{bkr} asking whether the domination number of an acyclic digraph is equal to its packing number.

\section{Notation and preliminaries}

Let $D=(V(D),A(D))$ be a digraph.
If $(u,v) \in A(D)$, then we say that $u$ \emph{dominates} $v$ or that $v$ \emph{is dominated} by $u$. A set $S\subseteq V(D)$ is a {\em dominating set} of $D$ if each vertex in $V(D) \setminus S$ is dominated by a vertex in $S$. The {\em domination number}, $\gamma(D)$, of $D$ is the smallest cardinality of a dominating set of $D$. A dominating set of $D$ of cardinality $\gamma(D)$ is a {\em $\gamma$-set} of $D$ or simply a {\em $\gamma(D)$-set}. An undirected graph $G$ can be considered as a digraph in which $A(G)$ is a symmetric binary relation on $V(G)$. The order of a (di)graph $G$ will be denoted by $n(G)$.

A vertex $u$ is an \emph{in-neighbor} of $v$ if $(u,v) \in A(D)$ and an \emph{out-neighbor} of $v$ if $(v,u) \in A(D)$. The \emph{open out-neighborhood} of $v$ is the set of out-neighbors of $v$ and is denoted by $N^+_D(v)$. The \emph{closed out-neighborhood} of $v$ is the set $N^+_D[v]$ defined by $N^+_D[v]=N^+_D(v)\cup \{v\}$.  In a similar manner one defines the \emph{open in-neighborhood} $N^-_D(v)$ of $v$ and the \emph{closed in-neighborhood} $N^-_D[v]$ of $v$. The \emph{in-degree} of $v$ is the number $|N^-_D(v)|$ and the \emph{out-degree} of $v$ is $|N^+_D(v)|$.
If the digraph $D$ is clear from the context, then we may omit the subscript $D$ of the above notations.

Let $G$ be an undirected graph. An {\em orientation of $G$} is a digraph in which every edge from $G$ is directed in one of the two possible directions. Formally, an orientation of $G$ is defined by a mapping $f:E(G)\to V(G)\times V(G)$, such that if $uv\in E(G)$, then $f(uv)\in \{(u,v),(v,u)\}$. We denote this orientation of $G$ by $G_f$, while we refer to $f$ as the {\em orienting mapping}. Note that by this definition, we can formulate $\D(G)$ as $\max\{\gamma(G_f):\, f\textrm{ is an orientating mapping of }G\}$.

The following observations are from~\cite[Observations 3 and 4]{ch-2012}.

\begin{lem}
\label{lem:induced-spanning} Let $G$ and $H$ be two graphs.
\begin{itemize}
\item[(i)] If $H$ is an induced subgraph of $G$, then $\D(G)\ge \D(H)$.
\item[(ii)] If $H$ is an spanning subgraph of $G$, then $\D(G)\le \D(H)$.
\end{itemize}
\end{lem}

We also recall the following result from~\cite[Lemma 3]{ch-2012}.
\begin{lem}\label{lem:partition}
Let $G$ be a graph, and let $V_1,\ldots,V_k$ be subsets of $V(G)$ such that $V(G)=V_1\cup\cdots\cup V_k$. Letting $G_i$ be the subgraph of $G$ induced by $V_i$ for all $i\in [k]$, we get $$\D(G)\le\sum_{j=1}^{k}{\D(G_i)}.$$
\end{lem}

Let $G$ be a graph.  We denote the {\em independence number}, {\em matching number}, {\em vertex cover number}, and {\em edge cover number} of $G$ by
$\alpha(G)$, $\alpha'(G)$, $\beta(G)$, and $\beta'(G)$, respectively.
It is well-known that $\alpha(G)+\beta(G)=n(G)$ and $\alpha'(G)+\beta'(G)=n(G)$ in any graph $G$. If $G$ is bipartite, then $\alpha'(G)=\beta(G)$ by the famous 
K\H{o}nig-Egerv\'{a}ry theorem.

We will use yet another result of Caro and Henning~\cite[Theorem 2(a) and Theorem 4(a)]{ch-2012}.

\begin{thm}\label{thm:bounds}
 If $G$ is a graph, then
\begin{itemize}
\item[(i)] $\D(G)\ge \alpha(G)$, and  equality holds if $G$ is bipartite.
\item[(ii)] $\D(G)\le n(G)-\alpha'(G)$.
\end{itemize}
\end{thm}

\section{Corona graphs}
\label{sec:cor}

Let $G$ and $H$ be two graphs. The {\em corona}, $G\odot H$, of $G$ and $H$ is the graph obtained from the disjoint union of $G$ and $n(G)$ copies of $H$, which we denote by $H_u$ for every $u\in V(G)$, and then joining each $u\in V(G)$ to all vertices of $H_u$.  The {\em join} of $G$ and $H$ is the graph $G+H$ obtained from the disjoint union of $G$ and $H$ by connecting each vertex of $G$ with each vertex of $H$.

We start with a simple observation about the join with $K_1$.

\begin{lem}\label{lem:join-universal}
If $G$ is a graph, then $\D(G+K_1)\in\{\D(G),\D(G)+1\}$.
\end{lem}
\begin{proof}
By Lemma~\ref{lem:induced-spanning}(i), $\D(G+K_1)\ge \D(G)$, since $G$ is an induced subgraph of $G+K_1$. Let $V(K_1)=\{u\}$, and consider an arbitrary orientation $(G+K_1)_f$. The restriction of this orientation to $G$ can be dominated by a set $S$ of at most $\D(G)$ vertices. Then $S\cup\{u\}$ dominates $(G+K_1)_f$. We conclude that $\D(G+K_1)\le \D(G)+1$.
\end{proof}

Clearly, $\D(K_2)=1$ and $\D(K_3)=2$. By~\eqref{eq:log} and Lemma~\ref{lem:induced-spanning}(i), for every $k\ge 2$ there exists $n_k\in\mathbb{N}$ such that $\D(K_{n_k})=k$ and $\D(K_{n_k+1})=k+1$.  Hence, the sequence of complete graphs contains an infinite subsequence for which the larger value in the conclusion of Lemma~\ref{lem:join-universal} is attained. (Since the bounds in~\eqref{eq:log} increase slowly, indices $n_k$ with $\D(K_{n_k})=\D(K_{n_k+1})$ appear more often.) More generally, let $G$ be an arbitrary graph. Combining Lemma~\ref{lem:induced-spanning}(i) and Eq.~\eqref{eq:log}, we infer that the sequence of graphs $(G+K_n)_{n\ge 1}$ contains a subsequence of graphs which attain the larger value in the conclusion of Lemma~\ref{lem:join-universal}.

The condition in the conclusion of Lemma~\ref{lem:join-universal} is crucial for determining the orientable domination number of corona graphs, as shown in Theorem~\ref{thm:corona} below.
Another, more explicit example of a sequence of graphs that enjoy the studied condition,  is obtained by the class of even paths.

Let $n$ be any even positive integer, and let $G_n=P_n+K_1$. We also let $V(G_n)=\{x,x_1,\ldots,x_n\}$, where $x$ is the universal vertex of $G_n$, and vertices $x_1,\ldots,x_n$ of the path of order $n$ are indexed in the natural order. Consider the orienting mapping $f$ of $G_n$ defined as follows: $f(x_ix_{i+1})=(x_i,x_{i+1})$ for $i\in [n-1]$ and
\begin{displaymath}
f(xx_j) = \left\{ \begin{array}{ll}
(x,x_j); & j\textrm{ odd},\\
(x_j,x); & j\textrm{ even}.
\end{array} \right.
\end{displaymath}
Note that $\D(P_n)=n/2$, and we claim that $\D(G_n)=n/2+1$. By Lemma~\ref{lem:partition}, $\D(G_n)\le \D(P_n)+\D(K_1)=n/2+1$. Let $D$ be a minimum dominating set of $(G_n)_f$. Since $V((G_n)_f) \setminus N^+[x]=\{x_i:\, i\textrm{ is even}\}$, and no two vertices of this set are dominated in $(G_n)_f$ by a single vertex, we infer that $x\in D$ implies $|D|=n/2+1$. Now suppose that $x\notin D$ and that $|D| \le n/2$.  Since $D$ is a dominating set of $(P_n)_f$ and $\gamma((P_n)_f)=n/2$, it follows that 
$|D|=n/2$.  On the other hand, $x_1 \in D$ since $x_1$ has no in-neighbor in $(P_n)_f$.  Moreover, $|D \cap \{x_{2i-1},x_{2i}\}|=1$ for each $i \in [n/2]$
since $D$ dominates $(P_n)_f$ and has cardinality $n/2$.  Since $x_1 \in D$, it follows immediately that $D=\{x_{2i-1} :\, i \in [n/2]\}$.  This is a contradiction
since $D$ does not dominate $x$.  Hence, $|D| \ge n/2+1$.

We conclude that $\D(P_n)=n/2$ and $\D(P_n+K_1)=n/2+1$.

\begin{thm}\label{thm:corona}
If $G$ and $H$ are two graphs, then
\begin{displaymath}
\D(G\odot H) = \left\{ \begin{array}{ll}
\D(H)n(G); & \D(H+ K_1)=\D(H),\\
\D(H)n(G)+\D(G); & \D(H+ K_1)=\D(H)+1.
\end{array} \right.
\end{displaymath}
\end{thm}
\begin{proof}
By Lemma~\ref{lem:join-universal}, $\D(H+ K_1)\in \{\D(H),\D(H)+1\}$.

First, consider the case when $\D(H+ K_1)=\D(H)$. Consider the subgraph $X$ of $G\odot H$ induced by $\bigcup_{u\in V(G)}{V(H_u)}$. Clearly, $X$ is the disjoint union of $n(G)$ copies of $H$, hence $\D(X)=n(G)\D(H)$. By Lemma~\ref{lem:induced-spanning}(i), $\D(G\odot H)\ge n(G)\D(H)$. Consider next the spanning subgraph $Y$ of $G\odot H$ obtained from $G\odot H$ by removing all the edges of $G$. Then, $Y$ is the disjoint union of $n(G)$ copies of $H\odot K_1$, and so $\D(Y)=n(G)\D(H\odot K_1)=n(G)\D(H)$. By Lemma~\ref{lem:induced-spanning}(ii), $\D(G\odot H)\le \D(Y)=n(G)\D(H)$.

Second, let $\D(H+ K_1)=k+1$, where $k=\D(H)$. Considering the induced subgraphs $G$ and $H_u$ for all $u\in V(G)$, Lemma~\ref{lem:partition} implies that $\D(G\odot H)\le \D(H)n(G)+\D(G)$.
For the reversed inequality, we construct an orientation $f$ of $G\odot H$ as follows. Let $h$ be an orientation of $H+ K_1$ such that $\gamma((H+ K_1)_h)=\D(H+ K_1)$, and let $g$ be an orientation of $G$ such that $\gamma(G_g)=\D(G)$. Now, the orientation $f$ of $G\odot H$ is defined by using $h$ on the edges of the subgraph $H'_u$ induced by $V(H_u)\cup\{u\}$ for each $u\in V(G)$, and by using $g$ on the edges of $G$ in $G\odot H$. Let $D$ be a minimum dominating set of $(G\odot H)_f$.
We claim that $|D\cap V(H'_u)|\ge k$. The inequality follows from the fact that all vertices in $V(H_u)$ are dominated only by vertices in $H'_u$, and $\gamma((H+ K_1)_h)=k+1$. Indeed, if $|D\cap V(H'_u)|\le k-1$, then we get a contradiction because $(D\cap V(H'_u))\cup\{u\}$ would be a dominating set of $(H_u')_h$ of size at most $k$.
Let $V_1=\{u\in V(G):\, |D\cap V(H_u')|=k\}$, and $V_2=V(G)-V_1$. We note that $u\in V_1$ implies $u\notin D$ for otherwise $(H_u')_h$ would be dominated by $k$ vertices. For the same reason, every $u\in V_1$ has to be dominated by a vertex in $D\cap V_2$. This implies that $V_2$ is a dominating set of $G_g$, which yields $|V_2|\ge \D(G)$. Therefore,
\begin{align*}
|D|&\ge |V_1|k+|V_2|(k+1)\\
&=k(|V_1|+|V_2|)+|V_2|\\
&=\D(H)n(G)+|V_2|\\
&\ge \D(H)n(G)+\D(G),
\end{align*}
and this completes the proof.
\end{proof}

\section{Cartesian products}
\label{sec:cart}

Recall that the {\em Cartesian product} of two graphs $G$ and $H$, denoted $G\cp H$, is the graph with vertex set $V(G\cp H) = V(G) \times V(H)$, where two vertices $(u, v)$ and $(x, y)$ are adjacent in $G\cp H$ if either $u=x$ and $vy \in E(H)$, or $v=y$ and $ux \in E(G)$.
The Cartesian product of digraphs is defined analogously; see~\cite{HIK}.

We first give general bounds for $\D(G\cp H)$.
\begin{thm}\label{thm:Cart} For any graphs $G$ and $H$,
\begin{align*}
\D(G\cp H) & \ge \max\{\D(G)\alpha(H), \alpha(G)\D(H)\} \\
\D(G\cp H) & \le \min\{\D(G)n(H), n(G)\D(H)\},
\end{align*}
and the bounds are sharp.
\end{thm}
\begin{proof}
Let $A$ be an $\alpha(H)$-set and let $G_f$ be an orientation of $G$ so that $\gamma(G_f) = \D(G)$.
Let $H_g$ be an arbitrary orientation of $H$.
Define the following mapping $h:E(G\cp H) \to V(G\cp H)\times V(G\cp H)$ by
\[h((u_i, v_j)(u_k, v_{\ell})) = \begin{cases} ((u_i, v_j),(u_k, v_j)); & \text{$j=\ell$ and $f(u_iu_k)=(u_i,u_k)$},\\
((u_i, v_j),(u_i, v_{\ell})); & \text{$i=k$ and $v_j \in A$},\\
((u_i, v_j),(u_i, v_{\ell})); & \text{$i=k$, $\{v_j, v_{\ell}\}\cap A=\emptyset$ and $g(v_jv_\ell)=(v_j,v_\ell)$.}
\end{cases}
\]
Considering the orientation $(G\cp H)_h$, for each $v \in A$, the only way to dominate vertices in $V(G)\times \{v\}$ is by vertices within $V(G)\times \{v\}$. Thus, $\gamma((G\cp H)_h)\ge \alpha(H)\D(G)$. Reversing the roles of $G$ and $H$, we have \[\D(G\cp H) \ge \max\{\D(G)\alpha(H), \alpha(G)\D(H)\}.\]

Partition the vertex set of $G\cp H$ into subsets $V(G)\times\{v\}$, for all $v\in V(H)$, and denote the subgraphs induced by these subsets by $G_v$. Applying Lemma~\ref{lem:partition}, we get
\[\D(G\cp H) \le \sum_{v\in V(H)} \D(G_v) = n(H)\D(G).\]

To see that the upper bound is sharp, consider bipartite graphs $G$ and $H$ such that $\alpha(H)=n(H)/2$ and $\alpha(G)=n(G)/2$. Thus, $V(H)$ can be partitioned into $\alpha$-sets $B$ and $B'$, and $V(G)$ can be partitioned into $\alpha$-sets $A$ and $A'$. Since $G\cp H$ is bipartite, we get $\D(G\cp H)=\alpha(G\cp H)$, by Theorem~\ref{thm:bounds}(i). Note that $(A\times B)\cup (A'\times B')$ is an $\alpha$-set of $G\cp H$, thus $$\D(G\cp H)=|A|\cdot|B|+|A'|\cdot|B'|=(|A|+|A'|)n(H)/2=n(G)n(H)/2.$$  On the other hand, $$\min\{\D(G)n(H), n(G)\D(H)\}=\min\{\alpha(G)n(H),\alpha(H)n(G)\}=n(G)n(H)/2,$$ which shows that the upper bound is indeed attained for such graphs $G$ and $H$.

The lower bound is also sharp and can be attained by taking the graphs $K_3$ and $P_3$. Let $X=P_3\cp K_3$. Set $V(P_3)=\{a,x,y\}$ with $E(P_3)=\{ax,ay\}$, and $V(K_3)=[3]$. Consider an arbitrary orientation $X_f$ of $X$. Assume without loss of generality that $((a,1),(a,2))\in A(X_f)$. If $((a,1),(x,1))\in A(X_f)$, then the vertex $(a,1)$ together with one vertex from each of the edges of the matching $\{(x,2)(x,3), (a,3)(y,3),(y,2)(y,1)\}$ gives us $\gamma(X_f)\le 4$. More generally, if a vertex of $X_f$ dominates a vertex in its copy of $K_3$ and a vertex in its copy of $P_3$, then we can find a dominating set of cardinality at most $4$. 
Using this argument, we can assume that $((x,1),(a,1)),((x,2),(x,1)), ((a,2),(x,2)), ((a,3),(a,2))\in A(X_f)$ and $((x,3),(a,3)), ((x,2),(x,3))\in A(X_f)$. Now, the vertex $(x,2)$ together with one vertex from each of the edges of the matching $\{(a,3)(y,3), (a,2)(y,2),(a,1)(y,1)\}$ yields $\gamma(X_f)\le 4$.
We have thus seen that $\D(P_3\cp K_3)\le 4$. The equality 
\begin{equation}
\label{eq:P3K3}
\D(P_3\cp K_3)=4\end{equation}
now follows from the fact that $\D(P_3\cp K_3) \ge  \alpha(P_3)\D(K_3)=2 \times 2$.
\end{proof}

We think that the lower bound in Theorem~\ref{thm:Cart} might not be attainable for graphs $G$ and $H$ with large enough order. We can verify that the lower bound is not attained in the case when $G$ and $H$ are nontrivial, connected bipartite graphs. The lower bound then reads $\alpha(G\cp H)\ge \alpha(G)\alpha(H)$, by Theorem~\ref{thm:bounds}(i). It was proved by Vizing that $\alpha(G\cp H)\ge \alpha(G)\alpha(H)+\min\{n(G)-\alpha(G),n(H)-\alpha(H)\}$ for any non-trivial graphs $G$ and $H$; see~\cite{HIK}. Thus, the lower bound in Theorem~\ref{thm:Cart} is not attained if both $G$ and $H$ are nontrivial, connected bipartite graphs.

Vizing's conjecture~\cite{v1968} from 1968 concerning the ordinary domination number of a graph in the Cartesian product of graphs is one of the main open problems in graph domination. The conjecture states that for any two graphs $G$ and $H$ the domination number $\gamma(G\cp H)$ of the Cartesian product of $G$ and $H$ is at least as big as the product $\gamma(G)\gamma(H)$ of their domination numbers. The inequality $\psi(G\ast H)\ge \psi(G)\psi(H)$, where $\psi$ is a graph invariant and $\ast$ is a product operation in graphs is often referred to as a Vizing-like inequality. 

We point out that if $G$ or $H$ is bipartite, then Theorem~\ref{thm:Cart} says
\[\D(G)\D(H) = \max\{\D(G)\alpha(H), \alpha(G)\D(H)\} \le \D(G\cp H).\]
We suspect a Vizing-like bound holds for $\D(G\cp H)$ which we formally pose as a problem.

\begin{prob} Is it true that $\D(G\cp H)\ge \D(G)\D(H)$ holds for any two graphs $G$ and $H$?
\end{prob}

In studying the above problem, it is important to note that given two directed graphs $D_1$ and $D_2$ it may be the case that  $\gamma(D_1\cp D_2) < \gamma(D_1)\gamma(D_2)$. For example, we know that $\gamma(\overrightarrow{C_3}) =2$ for the directed cycle $\overrightarrow{C_3}$. However, one can easily verify that $\gamma(\overrightarrow{C_3}\cp \overrightarrow{C_3}) = 3$. 

On the other hand, we can find an orientation showing that $\D(K_3\cp K_3) \ge 4$; see Fig.~\ref{fig:K_3boxK_3}. Let $G=K_3\cp K_3$.  Note that the black vertices are a dominating set of $G_f$. To see that $\gamma(G_f) \ge 4$, suppose there exists a dominating set $D$ of $G_f$ of cardinality $3$. Since each vertex of $G_f$ has out-degree $2$, then $D$ is an independent set in $K_3\cp K_3$. One can easily verify that each of the six $\alpha$-sets of $K_3 \cp K_3$ is not a dominating set of $G_f$. Thus, $\gamma(G_f) = 4 \le \D(G)$. Noting that $P_3\cp K_3$ is a spanning subgraph of $K_3\cp K_3$, combined with~\eqref{eq:P3K3}, and using Lemma~\ref{lem:induced-spanning}(ii), we infer $\D(K_3 \cp K_3)\le 4$. We conclude that
$\D(K_3\cp K_3)=4.$

\begin{figure}[h!]
\begin{center}
\begin{tikzpicture}[scale=2]

	\vertex (1) at (1, 0) [scale=1pt,fill=black]{};
	\vertex (2) at (2, 0) []{};
	\vertex (3) at (3, 0) []{};
	\vertex (4) at (1, 1) []{};
	\vertex (5) at (2, 1) []{};
	\vertex (6) at (3, 1) [fill=black]{};
	\vertex (7) at (1, 2) [fill=black]{};
	\vertex (8) at (2, 2) []{};
	\vertex (9) at (3, 2) [fill=black]{};
	\path
	(1) edge[->,very thick] (2)
	(2) edge[->,very thick] (3)
	(3) edge[->,very thick, bend right=30] (1)
	(4) edge[->,very thick, bend right=30] (6)
	(5) edge[->,very thick] (4)
	(6) edge[->,very thick] (5)
	(7) edge[->,very thick] (8)
	(8) edge[->,very thick] (9)
	(9) edge[->,very thick, bend right = 30] (7)
	(1) edge[->,very thick] (4)
	(4) edge[->,very thick] (7)
	(7) edge[->,very thick, bend right=30] (1)
	(5) edge[->,very thick] (2)
	(8) edge[->,very thick] (5)
	(2) edge[->,very thick, bend left=30] (8)
	(3) edge[->,very thick] (6)
	(6) edge[->,very thick] (9)
	(9) edge[->,very thick,bend right=30] (3)
	
	;

\end{tikzpicture}
\end{center}
\caption{The orientation $G_f$ of $G = K_3\cp K_3$}
\label{fig:K_3boxK_3}
\end{figure}

Next, we focus on the prism of a cycle. Theorem~\ref{thm:Cart} implies that $\D(C_n \cp K_2) \le n$.

\begin{prop} If $n \ge 2$, then $\D(C_n\cp K_2) = n$.
\end{prop}
\begin{proof} Let $V(C_n)=\{v_1,\ldots, v_n\}$ with vertices ordered in the natural order and let $V(K_2)=[2]$.
Let $(C_n\cp K_2)_h$ be the orientation of $C_n \cp K_2$ with arcs
\[\{((v_i, j),(v_{i+1}, j)):\, i\in[n],j\in [2], v_{n+1}=v_1\}\bigcup \{((v_i, 1),(v_i, 2)): i \in [n]\}.\]
Let $D$ be a dominating set of $(C_n\cp K_2)_h$, and $D_i=D\cap \{(v,i):\, v\in V(C_n)\}$, for $i\in [2]$.
Note that each vertex of the form $(v_i, 1)$ is dominated by a vertex in $D_1$, which implies that there exists no $i\in [n]$ such that $\{(v_i,1),(v_{i+1},1)\}\cap D= \emptyset$. Hence, the set $S$ of vertices in $\{(v,2):\, v\in V(C_n)\}$ that are not dominated by $D_1$ is independent. We derive that in order to dominate the vertices of $S$  one needs at least $|S|$ vertices from $D_2$, which implies $|D|=|D_1|+|D_2|\ge |D_1|+|S|=|V(C_n)|$.
\end{proof}

Given a graph $G$ let ${\rm bip}(G)$ denote the  maximum order of a bipartite induced subgraph of $G$.

\begin{prop}
\label{prp:prism}
If $G$ is a graph, then $${\rm bip}(G)\le \D(G\cp K_2)\le n(G).$$
\end{prop}
\begin{proof}
The upper bound follows from Theorem~\ref{thm:Cart}. For the lower bound consider the largest bipartite induced subgraph of $G$ and let $S$ and $T$ be the sets of its bipartition. Since $(S\times \{1\})\cup (T\times \{2\})$ is an independent set, it follows from Theorem~\ref{thm:bounds}(i) that
$\D(G\cp K_2)\ge |S|+|T|={\rm bip}(G)$.
\end{proof}

If $G$ is bipartite, then ${\rm bip}(G)=n(G)$, thus Proposition~\ref{prp:prism} implies $\D(G\cp K_2)=n(G)$.

\section{Lexicographic products}
\label{sec:lex}

Let $G$ and $H$ be two graphs. The {\em lexicographic product} $G\circ H$ of $G$ and $H$ is the graph with $V(G\circ H)=V(G)\times V(H)$ and two vertices $(x,y)$ and $(u,v)$ are adjacent in $G\circ H$ if either $xu\in E(G)$, or $x=u$ and $yv\in E(H)$.

\begin{prop}\label{prp:lexicographic}
If $G$ and $H$ are arbitrary graphs, then
$$\alpha(G)\D(H)\le \D(G\circ H)\le \min\{\D(G)n(H),\D(H)n(G)\}.$$
\end{prop}
\begin{proof}
To prove the lower bound consider the following orientation of $G\circ H$. Let $H_f$ be an orientation of $H$ such that $\gamma(H_f)=\D(H)$. For each $x \in V(G)$ consider the subgraph of $G\circ H$ induced by the set $\{(x,y):\, y\in V(H)\}$, and denote it by $H^x$. Clearly, $H^x$ is isomorphic to $H$. Orient the edges of $H^x$ 
consistent with the orienting mapping $f$. (That is, if $f$ maps $ab\in E(H)$ to the arc $(a,b)\in V(H)\times V(H)$, then let $(x,a)(x,b)\in E(G\circ H)$ be mapped to $((x,a),(x,b))\in V(G\circ H)\times V(G\circ H)$.) The edges among vertices with distinct first coordinates are oriented as follows.
Let $A$ be an $\alpha$-set of $G$. For each $u\in A$ and $v\in N_G(u)$ and any $h,h'\in V(H)$, let the edges $(u,h)(v,h')\in E(G\circ H)$ be oriented from $(u,h)$ to $(v,h')$.  The latter orientation yields that vertices in $H^x$, where $x\in A$, can only be dominated by vertices in $H^x$. This establishes the lower bound.

Since $G \cp H$ is a spanning subgraph of $G\circ H$, the upper bound follows immediately from Theorem~\ref{thm:Cart} and Lemma~\ref{lem:induced-spanning}(ii).
\end{proof}

If $G$ is bipartite, then Theorem~\ref{thm:bounds}(i) and Proposition~\ref{prp:lexicographic}
imply $$\D(G)\D(H)\le \D(G\circ H)\le \D(G)n(H),$$ which in turn implies that $$\D(G\circ \overline{K_s})= \D(G)\D(\overline{K_s})=s\cdot\D(G).$$ This shows that both bounds in Proposition~\ref{prp:lexicographic} are sharp.  In particular, the ``Vizing-like'' bound $\D(G)\D(H)\le \D(G\circ H)$ does not hold in general, as can be seen by taking $G=K_3=H$. Note that $K_3\circ K_3=K_9$.  In~\cite[p. 60]{cvy-1996} Chartrand et al. proved that  $\D(K_9)=3$, while $\D(K_3)^2=4$.

The situation when $G$ is not bipartite is much more complex. Proposition~\ref{prp:lexicographic} gives the upper bound $ks+s$ for $\D(C_{2k+1}\circ \overline{K_s})$, which we are able to improve as follows.
\begin{prop}\label{prp:oddcycle}
If $k\ge 2$ and $s\ge 2$, then $$ks\le \D(C_{2k+1}\circ\overline{K_s})\le ks+\left \lfloor \frac{s+1}{2}\right\rfloor.$$
\end{prop}
\begin{proof}
Throughout the proof we write $G=C_{2k+1}\circ\overline{K_s}$.
The lower bound follows directly from the lower bound in Proposition~\ref{prp:lexicographic}.
For the upper bound note that $G$ is a Hamiltonian graph. Hence, if $s$ is even, $n(G)$ is even, and thus $G$ has a perfect matching.
By Theorem~\ref{thm:bounds}(ii), $\D(G)\le n(G)-n(G)/2=ks+s/2=ks+\lfloor \frac{s+1}{2}\rfloor$. On the other hand, $s$ odd implies that $\alpha'(G)=\frac{n(G)-1}{2}$, and so $\D(G)\le\frac{(2k+1)s+1}{2}=ks+\lfloor \frac{s+1}{2}\rfloor$.
\end{proof}

\subsection{Generalized lexicographic products}
\label{sec:genlex}

The lexicographic product $G\circ H$ of graphs $G$ and $H$ can be described as the graph obtained from $G$ by replacing each vertex $u$ of $G$ with an isomorphic copy of $H$, say $H_u$, and adding all the edges between $H_u$ and $H_v$ whenever $uv\in E(G)$. This can be generalized by replacing each vertex $u$ of $G$ by an arbitrary graph $H_u$. If ${\cal H} = \{H_u:\, u\in V(G)\}$ is a collection of graphs associated with the vertices of $G$, then the graph constructed in this way is called the {\em generalized lexicographic product} and is denoted by $G\circ {\cal H}$; see~\cite{samodivkin-2021}.

Let $G$ be a graph and ${\cal H} = \{H_u:\, u\in V(G)\}$ be a collection of graphs associated with the vertices of $G$. The argument in the proof of Proposition~\ref{prp:lexicographic} yields the following lower bound and Lemma~\ref{lem:partition} gives the upper bound. If $X$ is an independent set of $G$, then
$$\sum_{u\in X} \D(H_u) \le \D(G\circ {\cal H}) \le \sum_{u\in V(G)} \D(H_u).$$

Chartrand et al.~\cite{cvy-1996} considered the orientable domination number in several families of graphs. In particular, they determined the orientable domination number of the graphs $K_{n,n,n}=K_3\circ \overline{K_n}$ for all positive integers $n$. In this subsection, we extend their result by considering arbitrary complete multipartite graphs. Note that $K_{n_1,\ldots, n_k}$ is the generalized lexicographic product with the first factor $K_k$ and the collection of $k$ edgeless graphs of order $n_1, \ldots, n_k$, respectively.

Given a non-decreasing sequence $n_1\le n_2\le\cdots\le n_k$ of positive integers, where $k\ge 2$, the vertices of the graph $K_{n_1,\ldots, n_k}$ can be partitioned into independent sets $A_1,\ldots, A_k$ such that $|A_i|=n_i$ for every $i\in [k]$ and for $1\le i< j\le k$, every vertex of $A_i$ is adjacent to every vertex of $A_j$.
We use this notation throughout the remainder of this subsection.  If $k=2$, then $K_{n_1,n_2}$ is a complete bipartite graph, and $\D(K_{n_1,n_2})=n_2$ by Theorem~\ref{thm:bounds}(i).

We start by proving a lower and an upper bound that hold for an arbitrary complete multipartite graph.

\begin{prop}
\label{prp:multipartite-general}
Given a non-decreasing sequence $n_1\le n_2\le\cdots\le n_k$ of positive integers, where $k\ge 2$, we have $$n_k\le\D(K_{n_1,\ldots, n_k})\le \max\{n_k,k\}.$$
\end{prop}
\begin{proof}
The lower bound follows from Theorem~\ref{thm:bounds}(i), since $n_k=\alpha(K_{n_1,\ldots, n_k})$. Set $X=K_{n_1,\ldots, n_k}$. For the proof of the upper bound, consider an orientation $X_f$ of $X$ such that $\gamma(X_f)=\D(X)$. If one of the sets $A_i$, where $i\in [k]$, is a dominating set of $X_f$, then $\D(X)\le n_i\le n_k$, as claimed. Otherwise, for each $i\in [k]$, there exists a vertex $x_j\in V(X) \setminus A_i$ such that $x_j$ dominates all vertices of $A_i$. Hence, the set $\{x_1,\ldots, x_k\}$ is a dominating set of $X_f$, giving $\D(X)\le k$, as claimed.
\end{proof}

The following result immediately follows from Proposition~\ref{prp:multipartite-general}, and resolves the orientable domination of a large class of complete multipartite graphs.

\begin{cor}
\label{cor:large-family-complete-multipartite}
If $n_1\le n_2\le\cdots\le n_k$ is a non-decreasing sequence  of positive integers, where $n_k\ge k\ge 2$, then $\D(K_{n_1,\ldots, n_k})=n_k=\alpha(K_{n_1,\ldots, n_k})$.
\end{cor}

Finally, we concentrate on complete tripartite graphs and extend the result on $\D(K_{n,n,n})$ from~\cite{cvy-1996}.

\begin{thm}
\label{thm:complete-tripartite}
If $1 \le n_1\le n_2\le n_3$, then
\begin{displaymath}
\D(K_{n_1,n_2,n_3}) = \left\{\begin{array}{ll}
n_3; & n_3\ge 3,\\
3; & n_1=n_2=n_3=2,\\
2; & \textrm{otherwise}.
\end{array} \right.
\end{displaymath}
\end{thm}
\begin{proof}
If $n_3\ge 3$, then $\D(K_{n_1,n_2,n_3})=n_3$, by Corollary~\ref{cor:large-family-complete-multipartite}. Hence, let $n_3\le 2$, and consider the following cases. To see that $\D(K_{2,2,2})=3$, first note that $\D(K_{2,2,2})\le 3$, by Proposition~\ref{prp:multipartite-general}. The orientation $(K_{2,2,2})_f$ of $K_{2,2,2}$ depicted in the following table, where we list the closed out-neighborhoods of each of the vertices of $K_{2,2,2}$:

\smallskip

\begin{tabular}{|c|c|c|c|c|c|c|}
\hline
$u$ &  $x_1$ & $x_2$ & $x_3$ & $x_4$ & $x_5$ & $x_6$ \\
\hline
$N^+[u]$ & $x_1,x_5,x_6$ & $x_2,x_4,x_6$ & $x_1,x_2,x_3$& $x_1,x_4,x_5$ & $x_2,x_3,x_5$ & $x_3,x_4,x_6$ \\
\hline
\end{tabular}\,,

\medskip

\noindent shows that $\gamma((K_{2,2,2})_f)=3$, since no two vertices dominate the oriented graph.

By Lemma~\ref{lem:induced-spanning}(i), $\D(K_{1,2,2}) \ge 2$ since $K_3$ is an induced subgraph of $K_{1,2,2}$.  Let $f$ be any 
orienting mapping of $K_{1,2,2}$ such that $\D(K_{1,2,2})= \gamma((K_{1,2,2})_f)$. If either of $A_2$ or $A_3$ is a dominating set of $(K_{1,2,2})_f$, 
then $\gamma((K_{1,2,2})_f)= 2$. Therefore assume that neither of $A_2$ or $A_3$ is a dominating set of $(K_{1,2,2})_f$.  
Thus $A_2$, resp. $A_3$, is dominated by a vertex $x_2$, resp. $x_3$.  Note that $x_2 \neq x_3$, for otherwise $\gamma((K_{1,2,2})_f)=1$, which is 
a contradiction.  In such a situation, it is easy to see that $\{x_2,x_3\}$ is a dominating set of $(K_{1,2,2})_f=2$. Therefore, $\D(K_{1,2,2})=2$.

Note that $K_{1,1,2}$ is isomorphic to $K_4-e$ while $K_{1,1,1}$ is isomorphic to $K_3$, and clearly $\D(K_4-e)=2=\D(K_3)$, which concludes the proof.
\end{proof}


\subsection{Domination and packing in an acyclic orientation of $C_{2k+1}\circ \overline{K_s}$}

The classical result of Meir and Moon~\cite{mm-1975} states that the domination number of a tree $T$ is equal to the 2-packing number of $T$, which is defined as the maximum number of pairwise disjoint closed neighborhoods in $T$. In~\cite{bkr} this result was extended to the context of digraphs. (See also
Mojdeh, Samadi and G.~Yero~\cite[Theorem 5]{msy-2019} where the special case of this result was proved for orientations of trees.) The extension uses the following notion -- the digraph version of a 2-packing.
A subset $P$ of $V(D)$ is a \emph{packing} of a digraph $D$ if there are no arcs joining vertices of $P$ and for every two vertices $x,y\in P$ there does not exist $v\in V(D)$
such that $\{(v,x),(v,y)\} \subseteq A(D)$.  The \emph{packing number}, $\rho(D)$, of $D$ is the cardinality of a largest packing in $D$.

The mentioned extension to digraphs~\cite[Theorem 1.2]{bkr} asserts that  if $T$ is a digraph whose underlying graph is a tree, then $\rho(T) = \gamma(T)$.
The authors then asked whether the result can be extended to all acyclic digraphs; see~\cite[Problem 1]{bkr}. Recall that a digraph is {\em acyclic} if it contains no directed cycles. In particular, an acyclic digraph contains no opposite arcs, hence it is an orientation of an undirected graph. The mentioned problem reads as follows:
Is $\rho(D)=\gamma(D)$ if $D$ is an acyclic digraph?

We now use the lexicographic product to answer this question in the negative. The following construction will be used. Consider the lexicographic product $C_{2k+1}\circ \overline{K_s}$, where $k\ge 2$ and $s\ge 2$. Let $V(C_{2k+1})=\{v_1,\ldots, v_{2k+1}\}$, and let $V(\overline{K_s})=\{w_1,\ldots,w_s\}$. Let $X=C_{2k+1}\circ \overline{K_s}$, and define the orientation using the following mapping $f:E(X)\to V(X)\times V(X)$. For every $i\in [2k]$ and $j,k\in [s]$, let $$f((v_i,w_j)(v_{i+1},w_k))=((v_i,w_j),(v_{i+1},w_k)),$$
and for every $j,k\in [s]$, let $$f((v_1,w_j)(v_{2k+1},w_k))=((v_1,w_j),(v_{2k+1},w_k)).$$

\begin{thm}\label{thm:problem-negative}
If $k\ge 2$, $s\ge 2$, and $f$ is the orienting mapping as defined above, then $$\gamma((C_{2k+1}\circ \overline{K_s})_f)=s+2k-2 \quad{ and }\quad \rho((C_{2k+1}\circ \overline{K_s})_f)=s+k-1.$$
\end{thm}
\begin{proof}
Set $X=C_{2k+1}\circ \overline{K_s}$ for this proof, and for each $i\in[2k+1]$ let $V_i=\{(v_i,w_j):\, j\in [s]\}$.

Let $D$ be a minimum dominating set of $X_f$. Note that the in-degree of each vertex from $V_1$ is $0$, which implies that these vertices belong to every dominating set of $X_f$, thus also to $D$. Hence the vertices from $V_1\cup V_2\cup V_{2k+1}$ are dominated by the vertices from $V_1$, where $V_1\subseteq D$. Consider an arbitrary set $V_i$, where $i\in \{3,\ldots, 2k\}$. In order to dominate vertices of $V_i$ either all of them lie in $D$ or $V_{i-1}\cap D\ne\emptyset$. If $V_i$ is dominated by a vertex $(v_{i-1},w_j)$ (from $V_{i-1}\cap D$), then we set $x_i=(v_{i-1},w_j)$. Otherwise, if $D\cap V_i=V_i$, then set $x_i=(v_{i},w_1)$. In this case, further set $x_{i+1}=(v_i,w_2)$ (note that $x_{i+1}$ dominates $V_{i+1}$). The resulting mapping, which assigns $x_i$ to $V_i$ is one-to-one, which implies that $V_1\cup\{x_3,\ldots,x_{2k}\}\subseteq D$, hence $\gamma(X_f)\ge s+(2k+1)-3=s+2k-2$. On the other hand, $V_1\cup\{(v_2,w_1),(v_3,w_1),\ldots, (v_{2k-1},w_1)\}$ is a dominating set of cardinality $s+2k-2$, which gives the first formula.

To establish the packing number of $X_f$, first note that in every packing $P$ of $X_f$, we have $|P\cap V_i|\le 1$ holds for all $i\in \{2,3,\ldots, 2k+1\}$. Indeed, if $\{(v_i,w_j),(v_i,w_k)\}\subseteq P$, then the arcs going from $(v_{i-1},w_1)$ to each of the two vertices imply that $P$ is not a packing. Now, if $|V_i\cap P|=1$, then $|V_{i-1}\cap P|=0=|V_{i+1}\cap P|$. Hence, $|P\cap (V_2\cup\cdots \cup V_{2k})|\le k$. Moreover, if $|P\cap (V_2\cup\cdots \cup V_{2k})|= k$, then $V_2\cap P\ne\emptyset$ 
and $V_{2k}\cap P\ne\emptyset$, which implies that $V_1\cap P=\emptyset = V_{2k+1} \cap P$. This yields $|P|=k<s+k-1$. On the other hand, if 
$|P\cap (V_2\cup\cdots \cup V_{2k})|\le k-1$, we have $|P| \le s+k-1$ because $|P \cap (V_1 \cup V_{2k+1})| \le s$.  
However, the set $V_1\cup \{(v_3,w_1), (v_5,w_1), \ldots, (v_{2k-1},w_1)\}$ is a packing of cardinality $s+k-1$, and is thus a maximum packing of $X_f$. 
Thus, $\rho(X_f)=s+k-1$.
\end{proof}

\section*{Acknowledgments}
We are grateful to one of the reviewers for carefully reading the paper and for a number of suggestions by which we improved the final version of the paper. 
This work was performed within the bilateral grant ``Domination in graphs, digraphs and their products" (BI-US/22-24-038). The authors also thank Trinity College for its support.  B.B.\ and S.K.\ were supported by the Slovenian Research Agency (ARRS) under the grants P1-0297, J1-2452, N1-0285 and J1-3002.

\end{document}